\documentclass[11pt,reqno]{amsart} 
\usepackage[utf8]{inputenc}
\usepackage{amsfonts,amsmath,amsthm,amssymb}
\usepackage[justification=centering]{caption}
\usepackage{color}
\usepackage{longtable}
\usepackage{fullpage}
\usepackage{enumitem}   
\usepackage{hyperref}
\usepackage{mathtools}
\usepackage{multirow}
\usepackage[dvipsnames]{xcolor}
\usepackage{thm-restate}
\usepackage{thmtools,tikz}
\usepackage{graphicx} 
\usepackage{comment}
\DeclareMathOperator{\PF}{PF}

\usepackage[justification=justified]{caption}

\usepackage{cite}

\theoremstyle{plain}
\newtheorem{thm}{Theorem}[section]
\newtheorem{lem}[thm]{Lemma}
\newtheorem{prop}[thm]{Proposition}
\newtheorem{cor}[thm]{Corollary}

\newtheorem{remark}[thm]{Remark}

\theoremstyle{definition}
\newtheorem{problem}{Problem}

\numberwithin{case}{thm}

\numberwithin{subcase}{case}

\theoremstyle{definition}
\newtheorem{defn}[thm]{Definition}

\newtheorem{ex}[thm]{Example}

\newcommand{\Lucky}{\mathsf{Lucky}} 
\newcommand{\lucky}{\mathsf{lucky}} 
\newcommand{\bump}[1]{#1^+}         
\newcommand{\nomax}[1]{#1^{\triangledown}}        
\newcommand{\sort}[1]{#1^{\uparrow}}
\newcommand{\disvec}{\mathsf{dis}}  

\title{On the Lucky and Displacement Statistics of\\Stirling Permutations}

\thanks{The authors would like to thank the GEMS workshop and AIM for the opportunity to collaborate.}

\author[Colmenarejo]{Laura Colmenarejo}
\address[L.~Colmenarejo]{Department of Mathematics, North Carolina State University, Raleigh, NC 27685}
\email{\textcolor{blue}{\href{mailto:lcomen@ncsu.edu}{lcomen@ncsu.edu}}}

\author[Dawkins]{Aleyah Dawkins}
\address[A.~Dawkins]{Department of Mathematical Sciences, George Mason University, Fairfax, VA 22030}
\email{\textcolor{blue}{\href{mailto:adawkin@gmu.edu}{adawkin@gmu.edu}}}

\author[Elder]{Jennifer Elder}
\address[J.~Elder]{Department of Computer Science, Mathematics and Physics, Missouri Western State University, St. Joseph, MO 64507}
 \email{\textcolor{blue}{\href{mailto:jelder8@missouriwestern.edu}{jelder8@missouriwestern.edu}}}
 
\thanks{J.~Elder was partially supported by an AWM Mentoring Travel Grant}

\author[Harris]{Pamela E. Harris}
\address[P.~E.~Harris]{Department of Mathematical Sciences, University of Wisconsin-Milwaukee, Milwaukee, WI 53211}
\email{\textcolor{blue}{\href{mailto:peharris@uwm.edu}{peharris@uwm.edu}}}
\thanks{P.~E.~Harris was partially supported by a Karen Uhlenbeck EDGE Fellowship.
}
\author[Harry]{Kimberly J.~Harry}
\address[K.~J. ~Harry]{Department of Mathematical Sciences, University of Wisconsin-Milwaukee, Milwaukee, WI 53211}
\email{\textcolor{blue}{\href{mailto:kjharry@uwm.edu}{kjharry@uwm.edu }}}

\author[Kara]{Selvi Kara} 
\address[S.~Kara]{Department of Mathematics, Bryn Mawr College, Bryn Mawr, PA 19010}
\email{\textcolor{blue}{\href{mailto:skara@brynmawr.edu}{skara@brynmawr.edu}}}

\author[Smith]{Dorian Smith}
\address[D.~Smith]{School of Mathematics, University of Minnesota, Minneapolis, MN 55455}
\email{\textcolor{blue}{\href{mailto:smi01055@umn.edu}{smi01055@umn.edu}}}

\author[Tenner]{Bridget Eileen Tenner}
\address[B.~E.~Tenner]{Department of Mathematical Sciences, DePaul University, Chicago, IL 60614}
\email{\textcolor{blue}{\href{mailto:bridget@math.depaul.edu}{bridget@math.depaul.edu}}}
\thanks{B.~E.~Tenner was partially supported by the NSF grant DMS-2054436.}


\begin{document}
\maketitle

\begin{abstract}
Stirling permutations are parking functions, and we investigate two parking function statistics in the context of these objects: lucky cars and displacement. 
Among our results, we consider two extreme cases: extremely lucky Stirling permutations (those with maximally many lucky cars) and extremely unlucky Stirling permutations (those with exactly one lucky car). 
We show that the number of extremely lucky Stirling permutations of order $n$ is the Catalan number $C_n$, and the number of extremely unlucky Stirling permutations is $(n-1)!$.
We also give some results for luck that lies between these two extremes. 
Further, we establish that the displacement of any Stirling permutation of order $n$ is $n^2$, and we prove several results about displacement composition vectors. 
We conclude with directions for further study.
\end{abstract}

\section{Introduction}

Throughout the paper, we let $\mathbb{N}\coloneqq\{0,1,2,3,\ldots\}$, $[i,j] \coloneqq\{i,i+1,\ldots, j-1,j\}$ for integers $i \le j$, and $[n] \coloneqq [1,n]$.
A \textit{parking function of length} $n$ is a tuple $\alpha=(a_1,a_2,\ldots,a_n)\in [n]^n$ whose rearrangement into weakly increasing order 
$\sort{\alpha}=(b_1,b_2,\ldots,b_n)$ satisfies $b_i\leq i$ for all $i\in[n]$.
For example, $\alpha_{1}=(1,3,1,5,6,3)$ and $\alpha_{2}=(4,4,1,2,2,3,3,1)$ are parking functions since their rearrangements are $\sort{\alpha}_1=(1,1,3,3,5,6)$ and $\sort{\alpha}_2=(1,1,2,2,3,3,4,4)$, respectively.
However, $\alpha_3=(2,2,3,1,6,6)$ is not a parking function since $\sort{\alpha}_3=(1,2,2,3,6,6)$ and $b_5=6 \nleq 5$.
We let $\PF_n$ denote the set of parking functions of length~$n$. 

We recall another equivalent interpretation of parking functions. 
\begin{defn}
    Consider a one-way street with exactly $n$ parking spots and a line-up of $n$ cars, each with a preferred spot that we record as a sequence $\alpha=(a_1,a_2,\ldots,a_n)\in [n]^n$.
    Car $i$ tries to park in its preferred spot  $a_i$. 
    If spot $a_i$ is available, then car $i$ parks there; otherwise, car $i$ continues down the street until it finds the first available spot to park and parks there.
    If all cars park, $\alpha$ is a \emph{parking function}.
\end{defn}
Throughout the paper, we use the language from the above definition when referring to parking functions.
Parking functions were introduced in ~\cite{konheim1966occupancy} by Konheim and Weiss in their study of hashing functions, and they established that $|\PF_n|=(n+1)^{n-1}$.
Since their introduction, much work has been done to study this family of combinatorial objects, their statistics, and numerous generalizations~\cite{unit_pf, unit_perm, inv_assort, knaple, DIAZLOPEZ20191674, flat_pf, Schumacher_Descents_PF}.

One such work, motivating the current, is that of Gessel and Seo on so-called ``lucky'' cars in parking functions \cite{GesselandSeo}. 

\begin{defn}\label{defn:lucky}
In a parking function $\alpha = (a_1, a_2, \ldots, a_n)$, car $i$ is said to be \textit{lucky} if it parks in its preferred spot, $a_i$. 
Given $\alpha\in\PF_n$, write
$\Lucky(\alpha)$ for the set of lucky cars in $\alpha$, and write $\lucky(\alpha)\coloneqq |\Lucky(\alpha)|$
for the number of lucky cars in $\alpha$. 
\end{defn}

For example,  $\Lucky(\alpha_1)=\{1\}$ for the parking function $\alpha_1=(1,1,\ldots,1)\in[n]^n$ satisfies, which is one of the ``unluckiest'' parking functions since the only car that parks in its preferred spot is car $1$. The lucky cars of a different parking function $\alpha_2=(n,1,1,\ldots,1)\in[n]^n$
are $1$ and $2$; that is, $\Lucky(\alpha_2) = \{1,2\}$ making $\alpha_2$ luckier than $\alpha_1$. 
Permutations of $[n]$ are the ``luckiest'' parking functions,
as every car gets to park in their preference; that is, if $\alpha$ is a permutation of $[n]$, then $\Lucky(\alpha)=[n]$ and $\lucky(\alpha)=n$. 

Gessel and Seo gave the following generating function for the lucky car statistic on parking functions~\cite{GesselandSeo}:
\begin{equation}\label{eq:GesselSeo}
  \sum_{\alpha \in \PF_n} q^{\lucky(\alpha)} = q\prod_{i=1}^{n-1} (i+(n-i+1)q).  
\end{equation}

Their work is related to rooted trees, and the proof relies on ordered set partitions and generating function techniques. More recently, it was established by Harris, Kretschmann, and Mart\'inez Mori, that the sequences $\alpha \in [n]^n$
(not necessarily parking functions) with exactly $n-1$ cars parking in their preferred spot  
 is equal to the total number of comparisons performed by the Quicksort algorithm given all possible orderings of an array of size $n$~\cite{harris2023lucky}. 

We recall another related statistic for parking functions known as  \textit{displacement}.

\begin{defn}
 Let $\alpha=(a_1,\ldots, a_n) \in [n]^n$ be a parking function. If car $i$ parks in spot $p(i)\in [n]$, then the \textit{displacement} of car $i$ is defined by $d(I)= p(i)-a_i$, the difference between where it parks and its preference. The \textit{total displacement} of $\alpha$ is given by  $d(\alpha)=\sum_i d(I)$.   In this way, cars with zero displacement are lucky cars, while any car whose displacement is positive is unlucky.
\end{defn}

Displacement has been widely studied in the literature.  Knuth established a bijection between the total displacement of parking functions and the number of inversions of labeled trees \cite{Knuth1998}. Beissinger and Peled constructed a bijection from parking functions to labeled trees that carries the displacement of parking functions to the external activity of trees \cite{Beissinger1997}. 
Yan studied recurrence relations for generalized displacement in $u$-parking functions \cite{ MR3408702}, and Aguillon, Alvarenga, Harris, Kotapati, Mori, Monroe, Saylor, Tieu, Williams found a bijection with the set of ideal states in the  Tower of Hanoi game and parking functions of length $n$ with displacement equal to one \cite{displacement_hanoi}. For more on the lucky statistic and displacement, we point the reader to the 2023 work of Stanley and Yin \cite{stanley2023enumerative}.

In this paper, we focus on a subset of parking functions consisting of Stirling permutations. 

\begin{defn}\label{defn:stirling permutations}
A permutation of  the multiset
$\{1,1,2,2,3,3,\ldots,n,n\}$ is a \textit{Stirling permutation of order $n$} 
if any value $j$ appearing between the two instances of $i$ satisfies $j>i$.
We write $Q_n$ to denote the set of Stirling permutations of order $n$, 
and we write $w \in Q_n$ as a word $w(1) w(2) \cdots w(2n)$. 
\end{defn}

For example, $123321 \in Q_3$, but $123231 \notin Q_3$ as 2 appears between the two instances of 3.

Gessel and Stanley studied Stirling permutations in {\cite{MR462961}}.  
The Stirling polynomial $Q_n(t)=\sum_{w\in Q_n}t^{des(w)}$, where $des(w)$ refers to the descent statistic in permutations, arises as the numerator for the generating function
\[\sum_{m\geq 0}S(m+n,n)t^m=\frac{Q_n(t)}{(1-t)^{2n+1}},\]
where $S(m+n,n)$ are the Stirling numbers of the second kind \cite[\href{https://oeis.org/A008277}{A008277}]{OEIS}. 
The Stirling numbers of the second kind $S(n,k)$ count the number of set partitions of $[n]$ consisting of exactly $k$ blocks~\cite{StanleyECVol1}.
In \cite{MR462961}, Gessel and Stanley showed $|Q_n|=(2n-1)!!$.
Since then, others have studied Stirling permutations to include work on enumerating descents and mesa sets, and other statistics related to generalized permutations~\cite{elizalde2022descents, ELIZALDE2021105429, MesaPaper, shankar2023enumeration}.

Stirling permutations of order $n$ can be interpreted as parking functions of length $2n$.
Given a Stirling permutation $w=w(1)w(2)\cdots w(2n)\in Q_n$, the weakly increasing rearrangement of $w$ is always $(1,1,2,2,3,3,\ldots, n,n)$, which satisfies the required inequality condition characterizing parking functions. Thus, we have $ w\in  \PF_{2n}$.
Henceforth, we abuse notation and drop parenthesis and commas when talking about parking functions, and instead adopt the one-line notation used in Stirling permutations.

Since Stirling permutations can be viewed as a subset of parking functions, it is natural to explore the behavior of parking function statistics within this subset. Our focus is on the displacement and lucky statistics of Stirling permutations. Furthermore, we investigate their admissible sets in the context of lucky statistics. Given a subset $S\subset[2n]$, we say $S$ is an \emph{admissible lucky set} (or just \emph{admissible}) if there exists $w\in Q_n$ such that $\Lucky(w)=S$. Otherwise, we say that $S$ is \emph{not admissible}.

Our paper is organized in the following way.
In Section~\ref{sec:extreme cases}, we consider two extreme cases  and prove the following:
\begin{itemize}[leftmargin=.3in]
    \item The so-called ``extremely unlucky'' Stirling permutations (those with exactly one lucky car) are enumerated by $(n-1)!$ (Theorem~\ref{thm:extremely unlucky}).
    \item The so-called ``extremely lucky'' Stirling permutations (those with exactly $n$ lucky cars) are enumerated by Catalan numbers (Theorem~\ref{thm:extremely lucky bijection with parentheses}).
\end{itemize}

In Sections~\ref{sec:admissible} and~\ref{sec:small cardinality}, we look between these extremes by considering possible admissible sets of Stirling permutations.
As a foray into this broad topic, we establish the following results:
\begin{itemize}[leftmargin=.3in]
\item If $S$ is an admissible lucky set for $Q_n$, then it is also admissible for all $Q_{m}$ with $m>n$. (Theorem~\ref{thm:n to n+1 admissible})
\item We provide conditions under which there are exactly two lucky cars ({Theorem}~\ref{thm:exactly two lucky cars}).
The set $\{1, n-1, 2n-2\}$ is admissible if and only if $n$ is even (Proposition~\ref{prop:admissible n even}).
\end{itemize}
In Section~\ref{sec:displacement}, we consider the displacement statistic 
of Stirling permutations and  prove the following results: 

\begin{itemize}[leftmargin=.3in]
    \item The displacement of any Stirling permutation is $n^2$ (Corollary~\ref{cor:displacement}).
    \item For any Stirling permutation, the number of cars displaced is between (inclusively) $n$ and $2n-1$ (Corollary~\ref{cor:bounds for m_D}).
    \item We show that extremely lucky  Stirling permutations are uniquely determined by their displacement composition, a tuple that describes the displacement of the unlucky cars (Theorem~\ref{thm:displacement_implies_unique} and Corollary~\ref{cor:displace_num}). 
\end{itemize}

We conclude the paper in Section~\ref{sec:open problems} with a selection of open problems.

\section{Extremely unlucky/lucky Stirling permutations}\label{sec:extreme cases}

The $\lucky$ statistic of Stirling permutations in $Q_n$ has a well-defined range, as we show in the upcoming Lemma~\ref{lem:lucky_range}. It is natural  
to look at the  
Stirling permutations that achieve the extreme values in this range. The main results of this section are characterizations and enumerations of those two families of Stirling permutations -- one enumerated by $(n-1)!$, and the other enumerated by the $n$th Catalan number.

\begin{lem}\label{lem:lucky_range}
  Let  $w = w(1) \cdots w(2n) \in Q_n$.   
    \begin{enumerate}
    \item[(a)]  Car 1 is always lucky. 
    \item[(b)]  If car $i$ is the first car with $w(i) = 1$, then $i \in  \Lucky(w)$.
      \item[(c)] $1 \leq \lucky(w)  \leq n$. 
    \end{enumerate}
 \end{lem}   

\begin{proof}
(a) Car $1$ is always lucky in each $ w\in Q_n$, as it is the first car to park when all the spots (including its desired spot) are available. Thus, $1\in \Lucky(w)$ for  every $w\in Q_n$.

(b) Let $i\in[n]$ be the smallest $i$ such that $w(i)=1$, i.e. car $i$ is the first car to prefer spot $1$. Then car $i$ parks in spot $1$.  Hence, $i \in \Lucky(w)$. 

(c) It follows from (a) that $ 1\leq \lucky(w)$ for every $w\in Q_n$. For the upper bound, note that each car prefers one of the first $n$ parking spots by the definition of Stirling permutations. Thus, at most $n$ cars get to park in their desired spots, meaning $ \lucky(w) \leq n$.
\end{proof}

In the context of unlucky cars, the following complementary observation is immediate yet useful.

\begin{remark}\label{rem:2n never lucky}
Any car whose preference is the second instance of a value in $w\in Q_n$ is unlucky. In particular, the last car $2n$ is always unlucky. 
\end{remark}
 
The main results of the next two subsections consider the extremes described by Lemmas~\ref{lem:lucky_range}(c). Namely, we establish that the number of ``extremely unlucky'' Stirling permutations (those with only one lucky car) is $(n-1)!$ (see Theorem~\ref{thm:extremely unlucky}) and that ``extremely lucky'' Stirling permutations (those with $n$ lucky cars) are counted by Catalan numbers (see Corollary~\ref{cor:extremely lucky are catalan}).

\subsection{Stirling permutations with extremely bad luck}

The lower bound in Lemma~\ref{lem:lucky_range}(c) suggests a sort of ``extremely bad luck'' for Stirling permutations: having exactly one lucky car.

\begin{defn}\label{defn:unlucky stirling}
    A Stirling permutation $w$ with $\Lucky(w) = \{1\}$ is an \emph{extremely unlucky} (Stirling) permutation.
\end{defn}

One necessary condition for having exactly one lucky car is that $w(1) = 1$. Similarly, we must also have $w(2) = 1$, since otherwise $w(2)$ would be lucky.

\begin{ex}
    There are two extremely unlucky Stirling permutations of order $3$: $112233$ and $112332$.
\end{ex}

We now prove a technical result that helps us characterize extremely unlucky Stirling permutations.

\begin{lem}\label{lem:early small values ruin luck}
    Let $w=w(1)w(2)\cdots w(2n) \in Q_n$ and $x \in [2n]$.
    If $w(i) < x$ for all $i \le x$, then $x \not\in \Lucky(w)$.
\end{lem}

\begin{proof}
    Assume that $w(i) < x$ for all $i \le x$, for some $x \ge 2$ (the statement is meaningless for $x < 2$). Thus the first $x$ cars all want to park in the first $x-1$ spaces. By the pigeonhole principle, at least one car is unlucky and has to take the next available space. Consequently, one of these initial $x$ cars parks in spot $x$. By the assumption, the first appearance of $w(j) = x$ occurs for $j > x$. Thus, car $j$ does not get to park in its preferred spot because spot $x$ has already been claimed. Hence $x \not\in \Lucky(w)$.
\end{proof}

\begin{ex}
    Let $w=122133 \in Q_3$. Consider, first, $x=3$, and note that $\{w(1), w(2), w(3)\} = \{1,2\}$. Indeed, $3 \not\in \Lucky(w)$, confirming the result of Lemma~\ref{lem:early small values ruin luck}. Now consider $x = 2$. Given $\{w(1),w(2)\} = \{1,2\}$, the assumption of Lemma~\ref{lem:early small values ruin luck} does not hold.  This still leaves open the possibility for car 2 to be unlucky. However, car 2 parks in its preferred spot 2, indicating that it is not unlucky after all. 
\end{ex}

In other words, if the first $x$ values of $w$ are all less than $x$, then $x$ is not lucky.
In fact, as the above example suggests, this is a necessary and sufficient condition to characterize the extremely unlucky Stirling permutations.
\begin{cor}\label{cor:extremely unlucky iff}
    A Stirling permutation $w \in Q_n$ is extremely unlucky if and only if for all $x \in [2,n]$, we have $w(i) < x$ for all $i \le x$.
\end{cor}

\begin{proof}
Let $w \in Q_n$ be a Stirling permutation. If for all $x \in [2,n]$, we have $w(i) < x$ for all $i \le x$, then $w$ is extremely unlucky by Lemma~\ref{lem:early small values ruin luck}. 

Now, suppose that $w$ is extremely unlucky. Recall that $2 \not\in \Lucky(w)$ and it means that $w(2) = 1$, giving the result for $x = 2$. Fix $x \in [3,n]$ and suppose, inductively, that for all $y \in [2,x-1]$, we have $w(i) < y$ for all $i \le y$. Suppose, for the purpose of obtaining a contradiction, that $w(x)\ge x$. For $w$ to be extremely unlucky, for each $y \in [2,x-1]$, it must be the case that the first $y$ cars have parked in spots $[1,y]$. Thus the first $x-1$ cars have parked in spots $[1,x-1]$, meaning that spot $w(x)$ is available for car $x$ to park in. Thus, car $x$ is lucky, which contradicts the assumption that $w$ is extremely unlucky.
\end{proof}

The following set is useful in our proof of the enumeration of the extremely unlucky Stirling permutations. 

\begin{defn}
    For a finite, nonempty set $S$ of integer numbers, we define $\nomax{S}$ to be the subset of $S$ consisting of everything but the largest element of $S$; that is,
    $$\nomax{S} \coloneqq S \setminus \{\max(S)\}.$$
\end{defn}

Next, we enumerate extremely unlucky Stirling permutations via a constructive proof in which we build extremely unlucky Stirling permutations of order $n$ by first positioning the two copies of $n$, then the two copies of $n-1$, and so on, always ensuring that the only lucky car is the first car whose preferred spot is $1$. 
The construction is demonstrated after the proof of the theorem, in Example~\ref{ex:extremely unlucky}.

\begin{thm}\label{thm:extremely unlucky}
    For any integer $n \ge 1$, there are $(n-1)!$ extremely unlucky Stirling permutations of order $n$.
\end{thm}
\begin{proof}
    Following Corollary~\ref{cor:extremely unlucky iff}, we can construct the extremely unlucky Stirling permutations of order $n$ as follows. 
    First, we position two (necessarily consecutive) copies of $n$, with the leftmost of those appearing as $w(i_n)$ for some $i_n \in \nomax{[n+1, 2n]}$.  There are $2n-(n+1) + 1 - 1 = n-1$ ways to do this. The rightmost $n$ necessarily appears at $w(\bump{i_n})$, where $\bump{i_n}$ is the smallest element larger than $i_n$ in the set $[n+1,2n]$; that is, $\bump{i_n} = i_n+1$.
    
    Next, we position the two copies of $n-1$. Because we are building a Stirling permutation, these are either consecutive or they immediately surround the two copies of $n$. Thus, it follows from Corollary~\ref{cor:extremely unlucky iff} that the leftmost of these appears as $w(i_{n-1})$ for some
    $$i_{n-1} \in \nomax{\Big([n, 2n] \ \setminus \ \{i_n, \bump{i_n}\}\Big)},$$
    and the rightmost $n-1$ necessarily have position $w(\bump{i_{n-1}})$, where $\bump{i_{n-1}}$ is the smallest element larger than $i_{n-1}$ in the set $[n, 2n] \setminus \{i_n, \bump{i_n}\}$. There are $2n - n + 1 - 2 - 1 = n-2$ choices for $i_{n-1}$. 
    
    We continue this process until we find two available positions for the leftmost $3$, one available position for the leftmost $2$, and finally insert the two $1$s as $w(1)$ and $w(2)$. For example, the leftmost $x$ appears as $w(i_x)$ for some
    $$i_x \in \nomax{\Big([x+1,2n] \ \setminus \ \{i_n,\bump{i_n}\} \ \setminus \ \{i_{n-1},\bump{i_{n-1}}\} \ \setminus \ \cdots \ \setminus \ \{i_{x+1},\bump{i_{x+1}}\}\Big)},$$
and thus there are $2n-(x+1)+1 - 2(n-x) - 1 = x-1$ choices for $i_x$. The second appearance of $x$ is at $w(\bump{i_x})$, where $\bump{i_x}$ is the smallest element larger than $i_x$ in the set
    $$[x+1,2n] \ \setminus \ \{i_n,\bump{i_n}\} \ \setminus \ \{i_{n-1},\bump{i_{n-1}}\} \ \setminus \ \cdots \ \setminus \ \{i_{x+1},\bump{i_{x+1}}\}.$$
     
    This process has constructed $(n-1)(n-2)\cdots 3\cdot 2 \cdot 1 = (n-1)!$ distinct Stirling permutations, which are exactly the extremely unlucky Stirling permutations of order $n$.
\end{proof}

We demonstrate the construction in the proof of Theorem~\ref{thm:extremely unlucky} with an example.

\begin{ex}\label{ex:extremely unlucky}
    Let $n = 6$.
    \begin{itemize}[leftmargin=.3in]
        \item The index $i_6$ can be any of the five values in $\nomax{[7,12]} = \{7,8,9,10,11\}$. Suppose $i_6 = 9$, and so $\bump{i_6} = 10$.
        \item The index $i_5$ can be any of the four values in $\nomax{\big([6,12] \, \setminus \, \{9,10\}\big)} = \{6,7,8,11\}$. Suppose $i_5 = 6$, and so $\bump{i_5} = 7$.
        \item The index $i_4$ can be any of the three values in $\nomax{\big([5,12] \, \setminus \, \{9,10\} \, \setminus \, \{6,7\}\big)} = \{5,8,11\}$. Suppose $i_4 = 8$, and so $\bump{i_4} = 11$.
        \item The index $i_3$ can be any of the two values in $\nomax{\big([4,12] \, \setminus \, \{9,10\} \, \setminus \, \{6,7\} \, \setminus \, \{8,11\}\big)} = \{4,5\}$. Suppose $i_3 = 5$, and so $\bump{i_3} = 12$.
        \item The index $i_2$ must be the one value in $\nomax{\big([3,12] \, \setminus \, \{9,10\} \, \setminus \, \{6,7\} \, \setminus \, \{8,11\} \, \setminus \, \{5,12\}\big)} = \{3\}$. 
        Thus we must have $i_2=3$ (in fact, $i_2$ is always 3),
        and hence $\bump{i_2} = 4$.
    \end{itemize}
    Thus, we have produced the Stirling permutation
    $$1 \ 1 \ 2 \ 2 \ 3 \ 5 \ 5 \ 4 \ 6 \ 6 \ 4 \ 3,$$
    which is, indeed, extremely unlucky.
\end{ex}

\subsection{Stirling permutations with extremely good luck}

In contrast to the extremely unlucky Stirling permutations of Theorem~\ref{thm:extremely unlucky}, we now consider the very fortunate ones. 
To begin, note that by Lemma~\ref{lem:lucky_range}(c), if $w  \in Q_n$, then $\Lucky(w) \subseteq [2n]$ with $\lucky(w)\leq n$. 
In other words, at most half of the cars are lucky.

\begin{defn}\label{defn:lucky stirling}
    A Stirling permutation $w \in Q_n$ for which $\lucky(w) =n$ is an \emph{extremely lucky} (Stirling) permutation.
\end{defn}

Note, in particular, that $w \in Q_n$ is extremely lucky if and only if lucky cars park in each of the first $n$ available parking spots.
As we will see, extremely lucky Stirling permutations are \emph{Catalan objects}; that is, they are enumerated by the Catalan numbers \cite[\href{https://oeis.org/A000108}{A000108}]{OEIS}.
\begin{defn}\label{defn:catalan}
    The \emph{Catalan numbers} are the sequence $\{C_0, C_1, \ldots\}$ defined by $C_n=\frac{1}{n+1}\binom{2n}{n}$. 
\end{defn}

Many families of combinatorial objects are Catalan objects, and we refer the interested reader to~\cite{stanley_2015} for a comprehensive book on the subject.
Our proof that extremely lucky permutations are another family of Catalan objects is established by giving a bijection to valid parenthesizations of $n$ pairs of parentheses; that is, expressions that contain $n$ pairs of parentheses that are correctly matched. 
It is well-known that valid parenthesizations are Catalan objects~\cite{StanleyECVol2}, a result which was proven by  Eug\`{e}ne Charles Catalan, although the sequence was known much earlier by Mingantu, a Mongolian mathematician~\cite{Larcombe}. 
In preparation for that main enumerative result, we first characterize extremely lucky Stirling permutations.

\begin{prop}\label{prop:extremely lucky second cars in order}
    A Stirling permutation $w \in Q_n$ is extremely lucky if and only if the subword of $w(1)\cdots w(2n)$ consisting of the second appearances of each number is the word $n\cdots 321$; that is, the second appearances of each number in $w$ occur in decreasing order.
\end{prop}

\begin{proof}
    Suppose first that there is some $x < n$ 
    for which the second appearance $x = w(j)$ appears to the left of the second appearance of $x+1$ in $w(1)\cdots w(2n)$. Since $w$ is a Stirling permutation, we must therefore have
    $$w = \cdots x \cdots x \cdots (x+1) \cdots (x+1) \cdots.$$
    But then $x+1$ would not be the parking spot of any lucky car because either car $j$  would park in spot $x+1$ or 
    some car $h$ with $h < j$ and $w(h) < w(j)$ would have parked in spot $x+1$, before any cars preferring spot $x+1$ have had a chance to do so. Thus, $w$ is not extremely lucky.

    For the other direction of the argument, suppose that the second appearances of each number in $w$ occur in decreasing order. Then certainly the first appearance of any  $x \in [1,n]$ in $w(1)\cdots w(2n)$ also appears to the left of the second appearance of any number in $[1,x-1]$. In fact, all the entries to the left of the first appearance of $x$ are either larger than $x$ or the first appearance of a number in $[1,x-1]$. Thus, spot $x$ is available for the first car preferring it for each $x \in [1,n]$. In other words, each car whose preference is the first appearance of a number in $[1,n]$ is a lucky car, and hence $\lucky(w)=n$. 
   \end{proof} 

Proposition~\ref{prop:extremely lucky second cars in order} characterizes extremely lucky Stirling permutations, and allows us to give a bijective correspondence between extremely lucky Stirling permutations of order $n$ and valid parenthesizations of $n$ pairs of parentheses.
For example, the $C_3=5$ valid parenthesizations of $3$ pairs of parentheses are: 
\[( \ ( \ ( \ ) \ ) \ ), \ \ ( \ ( \ ) \ ( \ ) \ ), \ \ ( \ ( \ ) \ ) \ ( \ ), \ \ ( \ ) \ ( \ ( \ ) \ ), \ \ \mbox{and}\ \ ( \ ) \ ( \ ) \ ( \ ).\]

\begin{thm}\label{thm:extremely lucky bijection with parentheses}
    Extremely lucky Stirling permutations of order $n$ are in bijection with valid parenthesizations of $n$ pairs of parentheses.
\end{thm}

\begin{proof}
    One direction of this argument is simple: given an extremely lucky $w \in Q_n$, replace the first occurrence of each $x$ by ``('' and the second by ``).'' Because $w$ is a Stirling permutation, this produces a valid parenthesization.

    Now, given a valid pairing of $n$ pairs of parentheses, replace each ``)'' by $n, \ldots,1$ in decreasing order from left to right. For each ``)'' replaced by $x$, replace its partner ``('' by $x$ too. This construction necessarily produces a Stirling permutation, and Proposition~\ref{prop:extremely lucky second cars in order} means that we get exactly those Stirling permutations that are extremely lucky.
\end{proof}

We demonstrate the argument in the proof of Theorem~\ref{thm:extremely lucky bijection with parentheses} with the following example.

\begin{ex}\label{ex:extremely lucky}
    By the bijection described above, the valid parenthesization
    $$ ( \ ) \ ( \ ) \ ( \ ( \ ) \ ( \ ( \ ) \ ) \ )$$
    is matched with the extremely lucky Stirling permutation 
    $$6 \ 6 \ 5 \ 5 \ 1 \ 4 \ 4 \ 2 \ 3 \ 3 \ 2 \ 1 \in Q_6.$$
\end{ex}

It is important to note that, as in the proof of Theorem~\ref{thm:extremely lucky bijection with parentheses}, 
any Stirling permutation can be paired with a valid parenthesization. 
Thus, many Stirling permutations could describe the same string of parentheses and among those, only one is extremely lucky. For instance, the Stirling permutations $112235546643$ and $665514423321$ both correspond to $( \ ) \ ( \ ) \ ( \ ( \ ) \ ( \ ( \ ) \ ) \ )$, but only the latter permutation is extremely lucky (cf.~Examples~\ref{ex:extremely unlucky} and~\ref{ex:extremely lucky}). 

\begin{cor}\label{cor:extremely lucky are catalan}
     For any positive integer $n$, the number of extremely lucky Stirling permutations of order $n$ is the Catalan number $C_n$.
\end{cor}

\section{Admissible lucky sets}\label{sec:admissible}
In this section, we give initial results towards characterizing the subsets of cars that can be lucky for Stirling permutations.  To this end, we recall the following definition from the introduction. 

\begin{defn}\label{def:addmissibleLucky}
    A subset $S\subseteq[2n]$ is an \emph{$n$-admissible lucky set} (or just \emph{admissible} if context is clear) if there exists $w\in Q_n$ such that $\Lucky(w)=S$. Otherwise, we say that $S$ is \emph{not $n$-admissible} (or \emph{not admissible}).  
\end{defn}

In what follows, we characterize when a set with two elements is an admissible lucky set and prove several properties of three-element admissible lucky sets. It remains an open problem to give a complete characterization of all admissible lucky sets. 

\begin{ex}\label{ex: admissible 1}
    For $n=4$, the set $S=\{1,3,6\}$ is a 4-admissible lucky set since  $\Lucky(w)=S$ for $w=3 3 1 4 4 2 2   1 \in Q_4$.
    In contrast, $S'=\{1,4,8\}$ is not $4$-admissible since there is no $w \in Q_4$ such that  $\Lucky(w)=S'$. Table~\ref{table:4-admissible sets} lists all $4$-admissible lucky sets. 
\end{ex}

\begin{table}[htbp]
\begin{tabular}{cccc}
\{1\} & \{1, 2\} & \{1, 2, 3\} & \{1, 2, 3, 4\}\\
& \{1, 3\} & \{1, 2, 4\} & \{1, 2, 3, 5\}\\
& \{1, 4\} & \{1, 2, 5\} & \{1, 2, 3, 6\}\\
& \{1, 5\} & \{1, 2, 6\} & \{1, 2, 3, 7\}\\
& \{1, 7\} & \{1, 2, 7\} & \{1, 2, 4, 5\}\\
& & \{1, 3, 4\} & \{1, 2, 4, 6\}\\
& & \{1, 3, 5\} & \{1, 2, 4, 7\}\\
& & \{1, 3, 6\} & \{1, 2, 5, 6\}\\
& & \{1, 3, 7\} & \{1, 2, 5, 7\}\\
& & \{1, 5, 6\} & \{1, 3, 4, 5\}\\
& & \{1, 5, 7\} & \{1, 3, 4, 6\}\\
& & & \{1, 3, 4, 7\}\\
& & & \{1, 3, 5, 6\}\\
& & & \{1, 3, 5, 7\}
\end{tabular}
\caption{The thirty-one admissible lucky sets for Stirling permutations of order $4$, arranged by size.}\label{table:4-admissible sets}
\end{table}

Determining the properties of admissible lucky sets is the focus of the remainder of this section.
We begin with a stability result, which establishes that if a set is $n$-admissible, then it is $m$-admissible
 for all $m\geq n$.
\begin{thm}\label{thm:n to n+1 admissible}
If $S$ is an $n$-admissible lucky set, then the sets $S$ and $S\cup \{2n+1\}$ are both $(n+1)$-admissible lucky sets.
\end{thm}

\begin{proof}
Suppose that $S$ is an $n$-admissible lucky set and that $w \in Q_n$ satisfies $\Lucky(w)=S$. Then the Stirling permutation $w'=w(n+1)(n+1) \in Q_{n+1}$, created by appending the value $(n+1)$ twice to the end of $w$, ensures that the last two cars whose preference is $n+1$ end up parking in spots $2n+1$ and $2n+2$, respectively. 
Thus, $\Lucky(w') = S$ as well. 

Now consider the Stirling permutation $w'' \in Q_{n+1}$ defined by $w''(i) \coloneqq w(i) + 1$ for all $1\leq i \le 2n$, and $w''(2n+1) = w''(2n+2) \coloneqq 1$. 
By construction, we have $\Lucky(w'') = \Lucky(w) \cup \{2n+1\}$, meaning that $S \cup \{2n+1\}$ is $(n+1)$-admissible, as desired.
\end{proof}

Next, we recall a consequence of Lemma~\ref{lem:lucky_range}(a) and Remark~\ref{rem:2n never lucky}, which involves elements of admissible lucky sets.
\begin{cor}\label{lem:2n_never_lucky}
If $S$ is $n$-admissible lucky set, then $1 \in S$ and $2n \not\in S$. 
\end{cor}

In fact, the preferred parking spot of the first lucky car  (car 1) determines the preferences of the initial unlucky cars. 
 
\begin{lem}\label{lem:initial unlucky preferences are forced}
    Let $w \in Q_n$ with $\lucky(w) \ge 2$, and let $x$ be the second-smallest element of $\Lucky(w)$. Then the first set of unlucky cars $2,3, \ldots, x-1$ park, in order, in spots
    $$w(1) + 1, \, w(1) + 2, \, \ldots, \, w(1)+x-2.$$
\end{lem}

\begin{proof}
    In order for cars $2, 3, \ldots, x-1$ to be unlucky, they must each prefer spots that have already been occupied when it is their turn to park. This implies that the preference of car $2$ must be $w(2) = w(1)$, ensuring that car $2$ parks in spot $w(1) + 1$. As the value $w(1)$ has now appeared twice, the value $w(3)$ (that is, the preference for car $3$) can only be $w(1) + 1$, and thus car $3$  parks in spot $w(1) + 2$. Similarly, regardless of its preference, car $4$ ends up parking in spot $w(1)+3$. Continuing in this manner, car $y$ parks in spot $w(1) + y-1$ for all $y \in [2,x-1]$. 
\end{proof}

\begin{remark}\label{rk: values of w(i)}
    Under the assumptions of Lemma~\ref{lem:initial unlucky preferences are forced}, note that $w(i)\geq w(1)$ for $i \in [2,x-1]$.
\end{remark}
The nature of Stirling permutations can impact the parity of certain relevant values, as we see in the following lemma. 

\begin{lem}\label{lem:stubborn stirling}
    Let $w \in Q_n$ and $k \in [1,n]$. Let $j$ be minimal such that $w(j) = k$; that is, $j$ is the first car to prefer spot $k$ in this parking function.  
    If $w(i)>w(j)$ for all $i\in[1,j-1]$, then $j$ must be odd.
\end{lem}

\begin{proof}
Let $w$, $k$, and $j$ be as described in the assumptions of the statement. Suppose, for the purpose of obtaining a contradiction, that $j$ is even.
Then $j-1$ is odd, so there exist some $h$ and $h'$ with $w(h) = w(h')$ and $h < j < h'$ as each number appears with 2 copies in $w$. 
We have $w(h) > w(j)$ by assumption, which contradicts the fact that $w$ is a Stirling permutation.
\end{proof}

In particular, the first car to prefer spot $1$ must be an odd-indexed car. In the spirit of the previous lemmas, we can deduce the parity of the second lucky car if it is in the second half of the queue.  

\begin{prop}\label{prop:if second lucky is big, then it's odd}
    Let $w \in Q_n$ with $\lucky(w) \ge 2$, and let $x$ be the second-smallest element of $\Lucky(w)$. If $x > n$, then $x$ is odd. 
\end{prop}

\begin{proof}
Suppose that $x > n$. 
By Lemma~\ref{lem:initial unlucky preferences are forced}, after the first $x-1$ cars have parked, the only unused parking spots are $[1,w(1)-1] \cup [w(1)+x-1,2n]$.
    Because $w(1)+x - 1 > 1 + n - 1 = n$, and cars can only prefer the first $n$ spots, the only way for car $x$ to be lucky would be if it prefers to park in spot $y$ for some $y < w(1)$. 
From this, Remark~\ref{rk: values of w(i)} and Lemma~\ref{lem:stubborn stirling} ensure that $x$ must be odd. 
\end{proof}

Having established several important foundational properties of admissible lucky sets, we spend the remainder of this section developing some of the more technical properties of these sets, which becomes relevant in our later examinations of admissible lucky sets of small cardinality. 
Many of these upcoming results focus on the second-smallest element of an admissible lucky set.

\begin{lem}\label{lem:second_lucky}
Let $w\in Q_n$ with $\lucky(w)\geq 2$, and let $x$ be the second-smallest element of $\Lucky(w)$.
    If $w(1)=1$, then $x \le w(x) \le n$. 
\end{lem}
\begin{proof}
Suppose that \(w(1) = 1\) and that car $x$ is the second lucky car.
By Lemma~\ref{lem:initial unlucky preferences are forced},
cars $2,3,\ldots, x-1$ park to the right of car~$1$, in spots $2,3,\ldots,x-1$, respectively.
Thus, if car \(x\) is to be lucky, then it must be in the first half of the queue and prefer one of the spots in the set \(\{ x, \dots, n\}\). In other words, \( x \le w(x) \le n\). 
\end{proof}

When the second-smallest element of $\Lucky(w)$ is sufficiently large, other lucky cars' preferences must be suitably small.

\begin{lem}\label{lem:first preference}
    Let $w \in Q_n$ with $\lucky(w) \ge 3$, and let $x$ be the second-smallest element of $\Lucky(w)$. 
    If $x\geq n$, then $w(1) \ge \lucky(w)$. 
    Moreover, the lucky cars park among the first $w(1)$ spots.
\end{lem}

\begin{proof}
   By Lemma~\ref{lem:initial unlucky preferences are forced}, for all $i \in [2,x-1]$, car $i$ parks in spot $w(1)+i-1$. Because $x \ge n$, spots $w(1)$ through $n$ are all occupied before  car $x$ parks. 
   Thus, to ensure there are $\lucky(w) - 1$ unused parking spots for the remaining lucky cars to park in, we must have that $w(1) - 1 \ge \lucky(w) - 1$; that is, $w(1) \ge \lucky(w)$. 
   Finally, because spots $w(1)+1$ through $n$ are occupied by unlucky cars, all lucky cars must park among spots $[1, w(1)]$.
 \end{proof}

Next, we characterize when the penultimate car in a Stirling permutation is lucky.

\begin{thm}\label{thm:2n-1lucky}
    For $w \in Q_n$, we have $2n-1 \in \Lucky(w)$ if and only if $w(2n-1) = 1$.
\end{thm}

\begin{proof}
First suppose that $w(2n-1) = 1$. If $w(2n) > 1$, then the value $1$ would appear between the two instances of $w(2n)$, meaning that $w$ is not a Stirling permutation. Thus, $w(2n) = 1$, and so car $2n-1$ is the first car to prefer spot $1$. Therefore, car $2n-1$ parks in spot 1, and hence it is lucky. 
    
    Now, suppose that $2n-1 \in \Lucky(w)$. Thus, car $2n-1$ is the first car to prefer $w(2n-1)$, and so $w(2n-1) = w(2n)$. If $w(2n-1) > 1$, then consider the two cars, say cars $i$ and $j$ with $i<j$, with preferred spot  $w(2n-1)-1$. If car $i$ is lucky, then either car $j$ or car $h$ for some $i < h < j$ parks in spot $w(2n-1)$. On the other hand, if car $i$ is not lucky then, for some $h \le i$, car $h$  parks in spot $w(2n-1)$. In either case, this contradicts the assumption that $2n-1 \in \Lucky(w)$. Therefore, we must have $w(2n-1) = 1$.
\end{proof}

\section{Admissible lucky sets of small cardinality}\label{sec:small cardinality}
In this section, we fully characterize and enumerate admissible lucky sets of size two and provide a first result toward a characterization for admissible lucky sets of size three.

\subsection{Admissible lucky sets of size two}
As discussed earlier $1 \in S$ for any admissible set $S$ (see Corollary~\ref{lem:2n_never_lucky}). In order to characterize admissible sets of size two, one needs to identify which cars can be the second lucky car.

\begin{thm}\label{thm:exactly two lucky cars}
    A set $S = \{h,  i\}$ with $h<i$ is an $n$-admissible lucky set if and only if $h = 1$ and either
    \begin{enumerate}
        \item $i>n$ and $i$ is odd, or
        \item $i \le n$.
    \end{enumerate}  
\end{thm}

\begin{proof}
    By Proposition~\ref{prop:if second lucky is big, then it's odd}, the sets described in the statement are the only possible lucky sets. For the other direction, we provide 
    Stirling permutations to demonstrate that all such $\{1,i\}$ can be obtained as lucky sets.  Let $w=112233 \cdots (n-1)(n-1)nn \in Q_n$. 
    \begin{itemize}
        \item Suppose $i$ is odd, i.e. $i=2k+1$. Let $\alpha$ be the Stirling permutation obtained from $w$ by  inserting $11$ in between $(k+1)(k+1)$ and $(k+2)(k+2)$, namely, 
        $$\alpha= 2\,2\,3\,3\,\cdots \,(k+1)\,(k+1)\,1\,1\,(k+2)\,(k+2)\,(k+3)\,(k+3)\,\cdots\, n\,n.$$
        Regardless of whether $i\leq n$ or $i>n$, one can verify that $\Lucky(\alpha) = \{1, i\}$.
        \item Suppose $i\leq n$ and $i$ is even, i.e. $i=2k$. Let $\alpha'$ be the Stirling permutation obtained from $w$ by inserting $k (2k)(2k)$ in between $(k-1)(k-1)$ and $(k+1)(k+1)$ and placing the remaining copy of $k$ at the end after $nn$, namely, 
        \begin{multline*}
        \alpha'=1\,1\,\cdots \,(k-1)\,(k-1) \, k \,(2k) \,(2k) \, (k+1)\,(k+1)\,\cdots \\ 
        \cdots \,(2k-1)\,(2k-1)(2k+1)\,(2k+1)\, \cdots\, n \,n \,k.         
        \end{multline*}
One can verify that $\Lucky(\alpha') = \{1, i\}$. \qedhere
    \end{itemize}
    \end{proof}

Enumeration of $n$-admissible lucky sets of size two follows from the above characterization. 

\begin{cor}
    For $n\geq 1$, there are 
    $\displaystyle{n-1+\left\lceil\frac {n-1}{2}\right\rceil}$ $n$-admissible lucky sets of size two.
\end{cor}

Next, we conclude that if  $n$ is the second-smallest element in $\Lucky(w)$ for $w \in Q_n$, then $\Lucky(w) = \{1,n\}$. 

\begin{lem}\label{lem:cardinality_2_with_n}
    Let $S$ be an $n$-admissible lucky set with 
    second-smallest element $n$. 
    \begin{enumerate}
        \item If $n$ is even, then $w(1)=1, w(n)=n$, and $|S|=2$. 
        \item If $n$ is odd and $w(1)=1$, then $w(n)=n$ and $|S|=2$. 
    \end{enumerate}
\end{lem}

\begin{proof}
Let $w\in Q_n$ with $\Lucky(w)=S$ and such that $1$ and $n$ are the two smallest elements in $S$. 

We first show that for $n$ even, $w(1)=1$. Suppose that $n$ is even and assume, for the purpose of obtaining a contradiction, that $w(1) > 1$. Then, by Lemma~\ref{lem:initial unlucky preferences are forced}, cars 2 through $n-1$ park in spots $w(1)+1,w(1)+2,\ldots, w(1) + n - 2$ where $w(1) + n - 2 \ge n$. So, lucky car $n$ must prefer a spot in the set $ [1, w(1)-1]$. Moreover, $w(n)$ must be smaller than every element of $\{w(1), \ldots, w(n-1)\}$.
By Lemma~\ref{lem:stubborn stirling}, $n$ must be odd. This is a contradiction. Thus, if $n$ is even, then $w(1) = 1$. 

Now assume that $w(1) = 1$, and make no assumptions on the parity of $n$. By Lemma~\ref{lem:initial unlucky preferences are forced}, the first $n-1$ cars park in the first $n-1$ spots, so the only way car $n$ can be lucky is if it prefers to park in spot $n$; i.e., if $w(n) = n$. There can be no additional lucky cars because the first $n$ spots are occupied by the first $n$ cars. Thus, $|S| = 2$.
\end{proof}

\subsection{On admissible lucky sets of size three}\label{subsec:three lucky cars}

For $n=4$, Table~\ref{table:4-admissible sets}
 gives the 11 distinct $4$-admissible lucky sets of cardinality three. 
In this section, we characterize certain families of three-element admissible lucky sets depending on the parity of $n$. It remains an open problem to characterize all three-element admissible lucky sets.

\begin{prop}\label{prop:admissible n even}
The set $\{1,n-1,2n-2\}$ is an $n$-admissible lucky set if and only if $n$ is even.
\end{prop}

\begin{proof}
First suppose that $w \in Q_n$ with $\Lucky(w) = \{1, n-1, 2n-2\}$. Assume $n > 3$, with the result being easy to check in small cases. 

Recall that the first car to prefer spot $1$ is always lucky, and by Lemma~\ref{lem:stubborn stirling}, it must be an odd-indexed car. Since $2n-2$ is always even, either car $1$ or car $n-1$, with $n$ even, must be the first car to prefer spot $1$. 

If $w(1) = 1$, by Lemma~\ref{lem:initial unlucky preferences are forced}, the only possible spots for the lucky car $n-1$ to prefer are $n-1$ or $n$. Since car $n$ is unlucky, it must have the same preference as a previous car, or its preferred spot is occupied. In order to satisfy the Stirling condition, we have $w(n-1) = w(n)$.  If $w(n-1) = n-1$, then the first $n$ cars park in the first $n$ spots, meaning that car $2n-2$ cannot possibly be lucky. On the other hand, if $w(n-1) = n$, then the only available spot for lucky car $2n-2$ to park is spot $n-1$. However, the preference of car $n+1$ is at most $n-1$. So, it ends up parking in the next available spot $n-1$, preventing car $2n-2$ from being lucky. Thus,  $w(1) = 1$ is not possible, and car $n-1$, with $n$ even, is the first car to prefer spot $1$. In other words, $\{1, n-1, 2n-2\}$ is $n$-admissible implies $n$ is even.

To prove the other direction of the result, it suffices to produce a Stirling permutation $w \in Q_{2k}$ for which $\Lucky(w) = \{1, 2k-1, 4k-2\}$.
Such a permutation is given below:
\[3\,3\,4\,4\,\cdots\,k\,k\,(k+1)\,(k+1)\,1\,(k+2)\,(k+2)\,\cdots\,(2k)\,(2k)\,2\,2\,1 \in Q_{2k}.\qedhere\]
\end{proof}

\section{Displacement statistic}\label{sec:displacement}

In this section, we study the displacement statistics of Stirling permutations. As discussed in the introduction, this statistic measures how far cars park from their preferred parking spots in a parking function.  

\begin{defn}\label{def:displacement}
    For any $w\in Q_n$, if car $i$ 
    parks in spot $p(i)\in[2n]$, then the \textit{displacement} of car $i$ is $d(i)\coloneqq p(i)-w(i)$. 
    The \textit{(total)  displacement} of $w\in Q_n$ is
    \[\displaystyle{d(w)=\sum_{i=1}^{2n} d(i)}.\]
\end{defn}

Because every Stirling permutation is a parking function, the displacement of a car is always nonnegative. Moreover, by definition, cars with displacement equal to zero are lucky cars, while cars whose displacement is positive are unlucky.

We begin by recalling a result from  \cite{elder2023costsharing} that for any $w\in Q_n$, the displacement of $w$ is invariant under permutations of the entries of $w$, even though the resulting strings are not necessarily Stirling permutations. 
This result is a special case of \cite[Lemma 3.1]{elder2023costsharing}, hence we omit its proof but state it formally below.

\begin{lem}\label{lem:constant displacement}\cite[Lemma 3.1]{elder2023costsharing}
    For any $w \in Q_n$, the displacement $d(w)$ is invariant under permuting the numbers in the one-line notation of $w$.
   In other words, for any permutation $\sigma$ of $[1,2n]$ 
    and any Stirling permutation $w\in Q_n$, let $\sigma(w)=w(\sigma(1))w(\sigma(2))\cdots w(\sigma(2n))$. Then $d(\sigma(w))=d(w)$.
\end{lem}

We now determine the displacement of Stirling permutations.

\begin{cor}\label{cor:displacement} 
For $w\in Q_n$, $d(w)=n^2$.   
\end{cor}

\begin{proof}
By Lemma~\ref{lem:constant displacement}, it suffices to compute $d(w)$ for the Stirling permutation $w=112233\cdots nn$. 
For this $w$, car $i$ parks in spot $i$, for all $i\in[2n]$. 
Thus, $d(w)=0+1+1+2+2+ \cdots+ (n-1)+(n-1)+n=n^2$, as desired.
\end{proof}

Definition~\ref{def:displacement} provides two versions of displacement: one for each car and one for the entire Stirling permutation. Corollary~\ref{cor:displacement} addresses the latter of these for Stirling permutations. Now we turn our attention to the former and investigate ``compositions" of $n^2$ that record the displacement data for each car when $w$ is in $Q_n$. 

\begin{defn}
    For $w\in Q_n$, we define the \textit{displacement composition} of $w$, denoted by $\disvec(w)$, as the $2n$-tuple encoding the displacements of each car: 
     \[\disvec(w)=(d(1),d(2),\ldots,d(2n))\in\mathbb{N}^{2n}.\]
\end{defn}

As discussed before, $d(i)=0$ if and only if car $i$ is lucky. In particular, the first entry of $\disvec(w)$ is always $0$.

\begin{ex} 
For $w = 11244233$, $\disvec(w) = (0,1,1,0,1,4,4,5)$.
\end{ex}

Recall that $1\leq \lucky(w)\leq n$ by Lemma~\ref{lem:lucky_range}(c). Thus, we can bound the number of nonzero entries in $\disvec(w)$ for any $w\in Q_n$.

\begin{cor}\label{cor:bounds for m_D}
For any $w\in Q_n$, we have that 
$1 \le |\{i \in [2n] : d(i) = 0\} | \le n$.
\end{cor}

Two distinct Stirling permutations may share the same displacement composition, indicating that displacement compositions do not uniquely identify a Stirling permutation. For instance, the displacement composition 
$(0, 1, 0, 1, 3, 4)$ corresponds to both 
$113322$ and $331122$. 
However, 
extremely lucky Stirling permutations are uniquely determined by their displacement compositions. 
That is,
no two distinct extremely lucky Stirling permutations have the same displacement composition. This is the main result in this section, and we start with a technical lemma used in the proof of that result.

\begin{lem}\label{lem:unlucky preference and displacement}
    Let $w \in Q_n$ be an extremely lucky permutation.  For $i\in[n]$:
    \begin{enumerate}
        \item[(a)] The $i$th unlucky car of $w$ prefers spot $n - i +1.$ 
        \item[(b)] The $i$th unlucky car of $w$ has displacement $2i-1.$
    \end{enumerate}
\end{lem}
\begin{proof}
    Let $w \in Q_n$ be an extremely lucky permutation. 
    Hence, $\lucky(w)=n$ and there are $n$ unlucky cars. 
    \begin{enumerate}
    \item[(a)] This is an immediate consequence of Proposition~\ref{prop:extremely lucky second cars in order}.

    \item[(b)] We proceed by induction. Suppose that car $j$ is the first unlucky car. It was established in (a) that car $j$ must prefer spot $n$. Since car $j$ is the second car to prefer spot $n$, and no other cars have yet been unlucky, then car $j$ parks in spot $n+1$. Thus, its displacement is $(n+1) - n = 2(1)-1$, as claimed.

    Now fix $i \ge 1$ and assume that the result holds for all $j \le i$. Thus, the $j$th unlucky car parks in spot $(n-j+1)+(2j-1) = n+j$ for all $j\le i$.
    Let car $k$ be the $(i+1)$th unlucky car.
    By (a), this car $k$ must prefer spot $n-i$.
    Since car $k$ is unlucky, 
    and by Proposition~\ref{prop:extremely lucky second cars in order}, we know that spots in $[n-i, n+i]$ are already occupied and the next available spot is $n+i+1$. 
    Thus, car $k$ parks in spot $n+i+1$. Hence, its displacement is $(n+i+1) - (n-i) = 2i+1 = 2(i+1)-1$, completing the proof. \qedhere
    \end{enumerate}
\end{proof}

We are now ready to establish the injectivity of the map $w \mapsto \disvec(w)$ where $w$ belongs to the set of extremely lucky Stirling permutations of order $n$.

\begin{thm}\label{thm:displacement_implies_unique} 
Let $w$ be an extremely lucky Stirling permutation. The displacement composition of $w$, $\disvec(w)$, uniquely determines  $w$. 
\end{thm}

\begin{proof}
Suppose that $w \in Q_n$ is an extremely lucky Stirling permutation, and consider its displacement composition $\disvec(w)$. Suppose that there exists another $w' \in Q_n$ such that $\disvec(w') = \disvec(w)$. 

The $0$ entries in a displacement composition correspond to the lucky cars in the parking function, so we must have $\Lucky(w) = \Lucky(w')$, and hence $w'$ must be extremely lucky as well. By Proposition~\ref{prop:extremely lucky second cars in order}, the values of $w'$ in each of the positions of nonzero entries of $\disvec(w)$ must be $n, n-1, \ldots, 1$, in order from left to right. Moreover, these must be the second appearances of each of these numbers in the one-line notation of $w'$.

Now, to obey the Stirling property, the first occurrence of $n$ in $w'$ must be in the rightmost unclaimed (lucky) position appearing to the left of the (already determined) second appearance of $n$ in $w'$. Similarly, the first occurrence of $n-1$ in $w'$ must be in the rightmost unclaimed position appearing to the left of the second appearance of $n-1$, and so on.

Each of these positions is entirely determined by $\disvec(w)$, and so in fact we must have $w = w'$.
\end{proof}

We demonstrate the result above in the following example. 
\begin{ex}
    The displacement composition $(0, 0, 0, 1, 3, 0, 0, 5, 7, 9)$ for $w \in Q_5$ tells us that $\Lucky(w) = \{1,2,3,6,7\}$. Moreover, Proposition~\ref{prop:extremely lucky second cars in order} tells us that
    $$ w = \underline{\ \phantom{1} \ } \ \underline{\ \phantom{1} \ } \ \underline{\ \phantom{1} \ } \ \ 5 \ \  4 \ \ \underline{\ \phantom{1} \ } \ \underline{\ \phantom{1} \ } \ \ 3 \ \ 2 \ \ 1.$$
    The Stirling condition tells us that the first appearance of $5$ in $w$ must be in position $3$. Then we see that the first appearance of $4$ in $w$ must be in position $2$. Similarly, the first position of $3$ in $w$ must be $w(7)$, the first position of $2$ must be $w(6)$, and the first position of $1$ must be $w(1)$. Therefore,
    $$w = \underline{\ 1 \ } \ \underline{\ 4 \ } \ \underline{\ 5 \ } \ \ 5 \ \ 4 \ \ \underline{\ 2 \ } \ \underline{\ 3 \ } \ \ 3 \ \ 2 \ \ 1.
    $$
\end{ex}

To conclude this section, we give a restatement of Corollary~\ref{cor:extremely lucky are catalan} in terms of displacement compositions.

\begin{cor}\label{cor:displace_num}
     The number of Stirling permutations whose displacement composition has exactly $n$ nonzero parts is $C_n$, 
    the $n$th Catalan number.
\end{cor}

\section{Open Problems}\label{sec:open problems}

Finally, we propose a selection of open problems related to lucky cars and the displacement statistic on Stirling permutations.

To begin with, recall that in Section~\ref{sec:small cardinality}, we gave several properties of three-element admissible lucky sets. The following problems remain unanswered. 

\begin{problem}
    Characterize and enumerate the three-element $n$-admissible lucky sets.

    More generally, for fixed $n$, how many $n$-admissible lucky sets are there?
\end{problem}

Equation~\eqref{eq:GesselSeo} gives a generating function for the lucky statistic on parking functions. 
Table~\ref{table:gen fun} provides analogous data for the generating function defined by the lucky statistic on Stirling permutations defined by
\[T_n(q)=\sum_{w\in Q_n}q^{\lucky(w)}.\]
Corollary~\ref{cor:extremely lucky are catalan}
implies that, for all $n\geq 1$, the leading coefficient of $T_n(q)$ is the Catalan number $C_n$. Theorem~\ref{thm:extremely unlucky} shows that the coefficient of $q$ in $T_n(q)$ is given by $(n-1)!$.
We thus pose the following questions.

\begin{problem}
    Determine a formula for $T_n(q)$ for any $n\geq 1$.
        Are the polynomials $T_n(q)$ unimodal? 
        Are they real rooted?

\end{problem}

\begin{table}[htbp]
{\renewcommand{\arraystretch}{1.5}
\centering
\begin{tabular}{|c||r|} \hline
\ \ $n$ \ \ & \multicolumn{1}{|c|}{$T_n(q)$}\\\hline
\hline
2 & $2q^2+q$ \\\hline
3 & $5q^3 + 8q^2 + 2q$ \\\hline
4 & $14q^4 + 49q^3 + 36q^2 + 6q$\\\hline
5 & $42q^5 + 268q^4 + 417q^3 + 194q^2 + 24q$\\\hline
6 & $132q^6 + 1374q^5 + 3876q^4 + 3665q^3 + 1228q^2 + 120q$\\\hline
7 & $429q^7 + 6752q^6 + 31231q^5 + 52353q^4 + 34675q^3 + 8975q^2 + 720q$\\\hline
8 & $1430q^8 + 32197q^7 + 227207q^6 + 620357q^5 + 710425q^4 + 355945q^3 + 74424q^2 + 5040q$\\\hline
\end{tabular}
}
\caption{The polynomials $T_n(q)$ for $2\leq n\leq 8$.}\label{table:gen fun}
\end{table}

The bound determined in Corollary~\ref{cor:bounds for m_D}, suggests the following avenue of research.

\begin{problem}\label{problem:existance}
Determine if, for $i\in[n,2n-1]$, there exists $w\in Q_n$ such that $\disvec(w)$ has exactly $i$ nonzero entries. Equivalently, can we always find a Stirling permutation with $i$ unlucky cars?
\end{problem}

Table~\ref{tab:displacement comps} provides some data in the case where $n=3$, listing all possible displacement compositions and the Stirling permutations that give rise to them.

It is shown in Theorem~\ref{thm:extremely unlucky} that there are $(n-1)!$ Stirling permutations in $Q_n$ whose displacement composition has exactly $1$ zero entry. Moreover, by Corollary~\ref{cor:displace_num}, there are $C_n$ Stirling permutations in $Q_n$ whose displacement composition has exactly $n$ nonzero parts. As a follow-up question to Problem~\ref{problem:existance}, we can thus try to interpolate between these two extreme cases,  enumerating the Stirling permutations whose displacement compositions have $k$ zero parts.

\begin{problem}
    For $k \in [2, n-1]$, determine
    $$\big| \{w \in Q_n: \disvec(w) \text{ has $k$ zero parts}\} \big|.$$
    Equivalently, how many $w \in Q_n$ have exactly $k$ lucky cars?
\end{problem}

It is also of interest to understand displacement compositions arising from general Stirling permutations.  

\begin{problem}\label{problem:displacement comps}
Let 
$\mathcal{D}_n\coloneqq\{\disvec(w):w\in Q_n\}$.
    Characterize and enumerate the elements of $\mathcal{D}_n$.
\end{problem}

On a more specific level, one can try to understand the fibers of the map $w \mapsto \disvec(w)$. In that direction, a related problem is to determine the number of Stirling permutations with a fixed displacement composition.
\begin{problem}
 Fix $\textbf{m} \in \mathcal{D}_n$. 
Characterize and enumerate the elements of the set
\[\mathcal{S}_n({\textbf{m}})\coloneqq\{w\in Q_n: \disvec(w)=\textbf{m}\}.\]
Put another way, what properties define the collection of Stirling permutations whose displacement compositions are equal to a given vector? Also, are there operations on Stirling permutations whose orbits define the fibers of that map?
\end{problem}

\begin{table}[htbp]
{\renewcommand{\arraystretch}{1.5}
    \centering
    \begin{tabular}{|c||c|}
\hline 
& $\disvec(w)$, where $w\in Q_3$\\[-.05in]
\ \ $k$ \ \ & has $k$ nonzero parts \\ \hline \hline
\multirow{5}*{$3$} &  $\disvec(123321)=(0,0,0,1,3,5)$\\
 & $\disvec(133221)=(0,0,1,0,3,5)$\\ 
&$\disvec(233211)=(0,0,1,3,0,5)$\\
&$\disvec(331221)=(0,1,0,0,3,5)$\\
&$\disvec(332211)=(0,1,0,3,0,5)$\\\hline
\multirow{8}*{$4$} 
&$\disvec(122331)=(0,0,1,1,2,5)$\\
&$\disvec(122133)=(0,0,1,3,2,3)$\\
&$\disvec(133122)=(0,0,1,1,3,4)$\\
&$\disvec(221331)=(0,1,0,1,2,5)$\\
&$\disvec(113322)=(0,1,0,1,3,4)$\\
&$\disvec(331122)=(0,1,0,1,3,4)$\\
&$\disvec(221133)=(0,1,0,3,2,3)$\\
&$\disvec(223311)=(0,1,1,2,0,5)$\\\hline
\multirow{2}*{$5$} 
&$\disvec(112233)=(0,1,1,2,2,3)$\\
&$\disvec(112332)=(0,1,1,1,2,4)$\\\hline
\end{tabular}
}
 \caption{Displacement compositions arising from Stirling permutations in $Q_3$, aggregated by the number of nonzero parts; i.e., by the number of unlucky cars.}
    \label{tab:displacement comps}
\end{table}

\bibliographystyle{plain}
\bibliography{bibliography}

\end{document}